\documentclass[]{interact}

\usepackage{epstopdf}
\usepackage[caption=false]{subfig}

\usepackage{color}
\usepackage[numbers,sort&compress]{natbib}
\usepackage{hyperref}
\usepackage[ruled,linesnumbered]{algorithm2e}
\bibpunct[, ]{[}{]}{,}{n}{,}{,}
\renewcommand\bibfont{\fontsize{10}{12}\selectfont}

\theoremstyle{plain}
\newtheorem{theorem}{Theorem}[section]
\newtheorem{lemma}[theorem]{Lemma}

\theoremstyle{definition}
\newtheorem{definition}[theorem]{Definition}
\newtheorem{example}[theorem]{Example}

\theoremstyle{remark}
\newtheorem{remark}{Remark}

\begin{document}


\title{An implementable descent method for nonsmooth multiobjective optimization on Riemannian manifolds}

\author{
\name{Chunming Tang\textsuperscript{a}, Hao He\textsuperscript{a}, Jinbao Jian\textsuperscript{b}\thanks{CONTACT J. B. Jian. Email: jianjb@gxu.edu.cn} and Miantao Chao\textsuperscript{a}}
\affil{\textsuperscript{a}College of Mathematics and Information Science, Center for Applied Mathematics of Guangxi, Guangxi University, Nanning, 530004, P. R. China; \textsuperscript{b}College of Mathematics and Physics; Guangxi Key Laboratory of Hybrid Computation and IC Design Analysis, Center for Applied Mathematics and Artificial Intelligence, Guangxi Minzu University, Nanning 530006, P. R. China}
}

\maketitle

\begin{abstract}
In this paper, an implementable descent method for nonsmooth multiobjective optimization problems on complete Riemannian manifolds is proposed. The objective functions are only assumed to be locally Lipschitz continuous instead of convexity used in the existing subgradient method for Riemannian multiobjective optimization. 
And the constraint manifold is a general manifold rather than some specific manifolds used in the proximal point method. 
The retraction mapping is introduced to avoid the use of computationally difficult geodesic.
A Riemannian version of the necessary condition for Pareto optimality
is proposed, which generalized the classical one in Euclidean space to the manifold setting. 
At every iteration, an acceptable descent direction is obtained by constructing a convex hull of some Riemannian $\varepsilon$-subgradients. And then a Riemannian Armijo-type line search is executed to produce the next iterate.
The convergence result is established in the sense that 
a point satisfying the necessary condition for Pareto optimality
can be generated by the algorithm in a finite number of iterations.
Finally, some preliminary numerical results are reported, which show that the proposed method is efficient.
\end{abstract}


\begin{amscode}
65K05; 90C30
\end{amscode}

\section{Introduction}

In the field of optimization, minimizing multiple objective functions at the same time is called multiobjective optimization. Usually, these objective functions are conflicting with each other, instead of having some common minimum points. For example, producers want to create higher value and make the cost as low as possible in production and manufacturing. Similar problems arise in many applications such as engineering design \cite{Engau}, management science \cite{Uttarayan Bagchi,M. Gravel},  environmental analysis \cite{Leschine,Fliege1}, etc. Due to its wide practical applications, multiobjective optimization has always been a hot topic, and a rich literature was produced; see the monographs \cite{Abraham,Deb,Sawaragi} and the references therein.

In most cases, the solution of multiobjective optimization problem is not a single point but a set of all optimal compromises, namely, the Pareto set. In traditional multiobjective optimization, one of the most popular methods is the scalarization approach \cite{Geoffrion}, whose idea is to convert a multiobjective problem to a single or a family of single objective optimization problems. However, in this method, the users need to select some necessary parameters because they are not known in advance, which may bring an additional cost. To overcome this shortcoming, there have other methods to solve such optimization problems, such as descent methods \cite{Fliege Steepest,El,Gebken B}, Newton-type methods \cite{Fliege Newton,Povalej Ž}, proximal point methods \cite{H. Bonnel,Bento6} and proximal bundle methods \cite{Makela1}, etc. These methods are almost all developed from a single objective optimization. In this paper, we are particularly interested in the case where the objective functions are not necessarily differentiable or convex. Recently, Gebken and Peitz \cite{Gebken B} proposed a descent method for locally Lipschitz multiobjective optimization, 
in which an acceptable descent direction for all objectives is selected as the element 
which has the smallest norm in the negative convex hull of certain subgradients of the objective functions.

In recent years, many traditional single objective optimization theories and methods have been extended from Euclidean space to Riemannian manifolds; see, e.g.,  \cite{P. A. Absil,Boumal,Hosseini1,Huang W,Hu,Hoseini,Sato,Zhu X}. Comparatively, for Riemannian multiobjective optimization, the relevant literature is very scarce, especially for nonsmooth cases. In \cite{Bento1} and \cite{Bento2}, a steepest descent method and an inexact version with Armijo rule for multiobjective optimization in the Riemannian context are presented, respectivley. Both methods require the objective functions to be continuously differentiable for partial convergence, and further assume that the objective vector function is quasi-convex and the manifold has nonnegative curvature for full convergence. In \cite{Bento3}, a proximal point method for nonsmooth multiobjective optimization on Hadamard manifold is developed. 
However, the application of this method is somewhat restricted, since as an important class of manifolds, 
the Stiefel manifold is not a Hadamard manifold.
In \cite{Bento4}, a subgradient-type method for Riemannian nonsmooth multiobjective optimization is presented, which requires the objective vector function to be convex. In addition, the geodesic needs to be generated at each iteration, which is generally a computationally daunting task, since it amounts to solving a second-order ordinary differential equation. 
In \cite{Eslami N}, a trust region method for Riemannian smooth multiobjective optimization problems is proposed, in which the objective functions are 
required to be second-order continuously differentiable for proving global convergence. As far as we know, numerical results are not reported in the existing literature for Riemannian nonsmooth multiobjective optimization.

Based on the above observations, the aim of this paper is to develop a 
practically and efficiently implementable method for nonconvex nonsmooth multiobjective optimization problems on general Riemannian manifolds.
More precisely, we propose a descent method for locally Lipschitz multiobjective optimization problems on complete Riemannian manifolds.
To the best of our knowledge, ours is the first method which is implementable and only requires that the objective functions are locally Lipschitz continuous on general manifolds. In particular, if the manifold constraint is dropped, the proposed method reduces to the method of \cite{Gebken B} in Euclidean space. However, our method is definitely not a straightforward extension of that in \cite{Gebken B}, since there are a lot of difficulties to be overcome. On one hand, we need to generalize the classical necessary condition of the Pareto optimal points for nonsmooth multiobjective optimization (see \cite{Makela2}) to the Riemannian setting. 
For this purpose, by combining with the retraction mapping, we transform the objective functions into the tangent space locally instead of considering them directly, {
	since some techniques and tools required for theoretical analysis are different from those in Euclidean space. For example, the Clarke generalized directional derivative in Euclidean space cannot be directly applied on manifolds. In fact, it is defined in the tangent space by resorting to a pullback function; see, e.g., \cite{Bentofo,Hosseini2}.}
On the other hand, the vector transport is introduced to handle linear combinations of subgradients in different tangent spaces.
And we need to carefully select the vector transport, aiming to ensure global convergence of the algorithm. 
More precisely, it should be satisfied the isometric property and the locking condition \cite{Huang W}. 
Based on such a vector transport, the Riemannian $\varepsilon$-subdifferential is then defined \cite{Hosseini2}.
We show that there exists a common descent direction for each objective function, which is just the element with the smallest norm in the set consisting of the negative convex hull of the Riemannian $\varepsilon$-subdifferentials of all objective functions.
Of course, it is generally not easy to compute the $\varepsilon$-subdifferentials of a nonsmooth function especially when its domain is a manifold. In order to save computational effort, inspired by the strategy adopted in the traditional methods \cite{Mahdavi-Amiri,Gebken B}, we use the convex hull of a special set to approximate the convex hull of Riemannian $\varepsilon$-subdifferentials of all objective functions. For this set, at the beginning, it consists of a single $\varepsilon$-subgradient of each objective function, then some new $\varepsilon$-subgradients are systematically computed and added to enrich the set, until the element with the smallest norm in its convex hull is an acceptable direction for each objective function. Furthermore, a Riemannian Armijo-type line search is executed to produce the next iterate.
Compared with \cite{Bento4}, by making use of the retraction mapping, the computation of the geodesic is avoided. The convergence result is established in the sense that an $(\varepsilon,\delta)$-critical point which is an approximation of the Pareto optimal point can be generated in a finite number of iterations. Finally, some preliminary numerical results are reported, which show that the proposed method is efficient.

This paper is organized as follows. In section \ref{sec2}, we recall some basic notations and definitions regarding Riemannian manifolds and locally Lipschitz function. In section \ref{sec3}, the necessary condition for Pareto optimality regarding locally Lipschitz multiobjective optimization problems is generalized to Riemannian manifolds, and the details of our method is presented. In section \ref{sec4}, we establish the convergence result of our method. In section \ref{sec5}, some preliminary numerical experiments are given.

\section{Preliminaries}\label{sec2}
Throughout of the paper, we denote by cl$S$ and conv$S$ the closure and the convex hull of a set $S$, respectively.
Letting $\mathcal{M}$ be a complete $d$-dimensional ($d\geq 1$) smooth manifold endowed with a Riemannian metric $\langle\cdot,\cdot\rangle_x$ on the tangent space $T_x\mathcal{M}$, we denote by $\|\cdot\|_x$ the norm which induced by Riemannian metric. We will often omit subscripts when they do not cause confusion and simply write $\langle\cdot,\cdot\rangle$ and $\|\cdot\|$ to $\langle\cdot,\cdot\rangle_x$ and  $\|\cdot\|_x$, respectively. The Riemannian distance from $x$ to $y$ is denoted by dist($x$,$y$), where the points $x,y\in\mathcal{M}$. Denote $\mathbb{B}(x,\sigma)=\{y\in\mathcal{M}|\,\,{\rm dist}(x,y)<\sigma\}$, and the tangent bundle by $T\mathcal{M}$. 

Firstly, we introduce the definition of locally Lipschitz functions on Riemannian manifolds; see, e.g., \cite{Hosseini1}.

\begin{definition}\label{def2.1}
	{
		Let $x\in\mathcal{M}$. If there are constant $L>0$ and a neighborhood $\mathcal{N}\subset\mathcal{M}$ of $x$ such that $f:\mathcal{M}\rightarrow \mathbb{R}$ satisfies} 
 \begin{equation*}
    {
    	|f(z)-f(y)|\leq L{\rm dist}(z,y),\,\, {\rm for} \,\, {\rm all} \,\,z, y\in \mathcal{N},}
 \end{equation*} 
 we say that $f$ is Lipschitz continuous near $x$ with the constant $L$. Furthermore, if for all $x\in\mathcal{M}$, $f$ is Lipschitz continuous near $x$, then we say that $f$ is a locally Lipschitz (continuous) function on $\mathcal{M}$.

\end{definition}

Now we consider the Riemannian nonsmooth multiobjective optimization problem:
\begin{eqnarray}\label{pro1}
	\min_{x\in\mathcal{M}}F(x):=
	\left (
	\begin{array}{ll}
		f_1(x)\\
		\,\,\,\,\,\,\vdots\\
		f_m(x)
	\end{array}
	\right ),
\end{eqnarray}
where $F:\mathcal{M}\rightarrow \mathbb{R}^m$ is called objective vector function, and the components $f_i:\mathcal{M}\rightarrow \mathbb{R}$ for $i\in\{1,\cdots,m\}$ are called objective functions, which are assumed to be locally Lipschitz continuous on $\mathcal{M}$. Clearly, the concept of optimality for real-valued function no longer applies, since the objective function of problem (\ref{pro1}) is vector valued. So we introduce the following so-called Pareto optimality (see \cite[Ch. 2]{Ehrgott}), and 
our aim is to find (approximate) Pareto optimal points on manifolds.

\begin{definition}\label{def2.2}
	Let $x\in\mathcal{M}$. If there is no $y\in\mathcal{M}$ such that
	$$f_i(y)\leq f_i(x)\,\,\,\,\forall i\in\{1,\cdots,m\}\,\,{\rm and}\,\,f_j(y)< f_j(x)\,\,\,\,{\rm for\,\,some}\,\,j\in\{1,\cdots,m\},$$
	then we say that $x$ is a Pareto optimal point for the problem (\ref{pro1}). Pareto set is the set consisting of all Pareto optimal points.

\end{definition}

For optimization problems posed on nonlinear manifolds, the concept of retraction can help us to develop a theory which is similar to line search methods in $\mathbb{R}^n$; see \cite[Def. 4.1.1]{P. A. Absil}.

\begin{definition}\label{def2.3}
	A smooth mapping $R:T\mathcal{M}\rightarrow \mathcal{M}$ is called a retraction on a manifold $\mathcal{M}$ if it has the following properties:
	
	(i) $R_x(0_x)=x$, where $0_x$ denotes the zero element of $T_x\mathcal{M}$;
	
	(ii) with the canonical identification $T_{0_x}T_x\mathcal{M}\simeq T_x\mathcal{M}$, $R_x$ satisfies
	$${\rm D}R_x(0_x)={\rm id}_{T_x\mathcal{M}},$$
	where $R_x$ is the restriction of $R$ to $T_x\mathcal{M}$, and
	${\rm id}_{T_x\mathcal{M}}$ denotes the identity mapping on $T_x\mathcal{M}$.
\end{definition}

It is further assumed that there is a constant $\kappa$ such that
\begin{equation}\label{eq2.1}
	{\rm dist}(R_x(\xi_x),x)\leq \kappa\|\xi_x\|
\end{equation}
for all $x\in \mathcal{M}$ and $\xi_x\in T_x\mathcal{M}$. {
	Intuitively, the inequality (\ref{eq2.1}) implies that the distance between $x$ and $R_x(\xi_x)$ is controlled by $\|\xi_x\|$}. Furthermore, it can ensure that the point $R_x(\xi_x)$ is in the neighborhood of $x$ when $\|\xi_x\|$ is small enough. This does not constitute any restriction in most cases of interest; see \cite{Hosseini1}.
Denote $B_R(x,r)=\{R_x(\eta_x)|\,\,\|\eta_x\|<r\}$, which is an open ball centered at $x$ with radius $r$ if the retraction $R$ is the exponential mapping; see \cite{Hosseini2}.

Since we will work in different tangent spaces, it is necessary to introduce the concept of vector transport (see \cite[Def. 8.1.1]{P. A. Absil}). As its name implies, it serves to move vectors in different tangent spaces to the same tangent space. Particularly, parallel translation along geodesics is a vector transport.

\begin{definition}\label{def2.4}
	A smooth mapping $\mathcal{T}:T\mathcal{M}\oplus T\mathcal{M}\rightarrow T\mathcal{M}:(\eta_x,\xi_x)\mapsto\mathcal{T}_{\eta_x}(\xi_x)$ is said to be a vector transport associated to a retraction $R$ if for all $(\eta_x,\xi_x)$, the following conditions hold:
	
	
	(i) $\mathcal{T}_{\eta_x}:T_x\mathcal{M}\rightarrow T_{R_x(\eta_x)}\mathcal{M}$ is a linear map;
	
	(ii) $\mathcal{T}_{0_x}(\xi_x)=\xi_x$ for all $\xi_x\in T_x\mathcal{M}$.
\end{definition}

Briefly, if $\xi_x\in T_x\mathcal{M}$ and $R(\eta_x)=y$, then $\mathcal{T}_{\eta_x}$ transports vector $\xi_x$ from the tangent space of $\mathcal{M}$ at $x$ to the tangent space at $y$. In order to obtain convergence results, the following conditions are also required.

\begin{description}
	\item[$\bullet$] The vector transport $\mathcal{T}$ preserves inner products, i.e.,
	\begin{equation}\label{inner}
		\langle\mathcal{T}_{\eta_x}(\xi_x),\mathcal{T}_{\eta_x}(\zeta_x)\rangle=\langle\xi_x,\zeta_x\rangle.
	\end{equation}

	\item[$\bullet$] The following locking condition is satisfied for $\mathcal{T}$, i.e.,
	\begin{equation}\label{locking}
		\mathcal{T}_{\xi_x}(\xi_x)=\beta_{\xi_x}\mathcal{T}_{R_{\xi_x}}(\xi_x),\quad \beta_{\xi_x}=\frac{\|\xi_x\|}{\|\mathcal{T}_{R_{\xi_x}}(\xi_x)\|},
	\end{equation}
	where$$\mathcal{T}_{R_{\eta_x}}(\xi_x)={\rm D}R_x(\eta_x)[\xi_x]=\frac{\rm d}{{\rm d}t}R_x(\eta_x+t\xi_x)|_{t=0}.$$
\end{description}

The above conditions are satisfied with $\beta_{\xi_x}=1$ if the retraction and vector transport are selected as the exponential map and parallel transport, respectively; see \cite{Huang W} for more details. 

In addition, it is necessary to introduce the notion of injectivity radius for $R_x$, since we need to transport subgradients from tangent spaces at some points lying in the neighborhood of $x\in\mathcal{M}$ to the tangent space at $x$; see \cite{Hosseini2}.

\begin{definition}\label{def2.5}
	{
		Let
	$$\iota(x):={\rm sup}\{\varepsilon>0\,| \,\, R_x:B(0_x,\varepsilon)\rightarrow B_R(x,\varepsilon) {\rm \,\, is\,\, injective}\},$$
	where $B(0_x,\varepsilon)=\{\xi_x\,|\,\,\|\xi_x\|<\varepsilon\}\subset T_x\mathcal{M}$. Furthermore, for this retraction $R$, the injectivity radius of $\mathcal{M}$ is defined as
	$$\iota(\mathcal{M}):=\inf_{x\in\mathcal{M}}\iota(x).$$
 }
\end{definition}

{
	For simplicity, 
the following intuitive notations are used:
$$\mathcal{T}_{x\rightarrow y}(\xi_x):=\mathcal{T}_{\eta_x}(\xi_x),\,\, {\rm and}\,\,\mathcal{T}_{x\leftarrow y}(\xi_x):=(\mathcal{T}_{\eta_x})^{-1}(\xi_y)\,\, {\rm whenever}\,\, y=R_x(\eta_x).$$
}

\begin{remark}\label{re1}
	(i) As in usual we assume that $\iota(\mathcal{M})>0$, and that an explicit positive lower bound of $\iota(\mathcal{M})$ is available, which will be used as an input of the algorithm. In fact, when $\mathcal{M}$ is compact, we at least know that $\iota(\mathcal{M})>0$.
	(ii) Clearly, $\mathcal{T}_{x\leftarrow y}(\xi_x)$ is well defined for all $y\in B_R(x,\iota(x))$, especially, for all $y\in B_R(x,\iota(\mathcal{M}))$. In what follows, it will always be ensured that $\mathcal{T}_{x\leftarrow y}(\xi_x)$ is well defined when we use it.
\end{remark}

We close this section by recalling the notion of Riemannian subdifferential, which is an extension of the classical Clarke subdifferential; see \cite{Hosseini2}.
If $X$ is a Hilbert space, and $\phi$ is a locally Lipschitz function defined from $X$ to $\mathbb{R}$, the Clarke generalized directional derivative of $\phi$ at $x$ in direction $v$ is defined as
$$\phi^\circ (x;v)=\limsup_{y\rightarrow x,\,\,t\downarrow 0}\frac{\phi(y+tv)-\phi(y)}{t},$$
and the Clarke subdifferential of $\phi$ at $x$ is then given by
$$\partial \phi(x):=\{\xi\in X\mid\langle \xi,v\rangle\leq \phi^\circ (x;v)\,\,{\rm for}\,\,{\rm all}\,\,v\in X\}.$$

\begin{definition}\label{def2.6}
	Let $f:\mathcal{M}\rightarrow \mathbb{R}$ be a locally Lipschitz function and for $x\in \mathcal{M}$, denote the pullback $\hat{f}_x$ of $f$ through $R$ at $x$ by $f\circ R_x$ (i.e., $\hat{f}_x=f\circ R_x$), where $R$ is a retraction. {The  Clarke generalized directional derivative of $f$ at $x$ in direction $p\in T_x\mathcal{M}$ is defined as} 
	$$f^\circ (x;p)=\hat{f}^\circ _x(0_x,p),$$ 
	where $\hat{f}^\circ _x(0_x,p)$ is the Clarke generalized directional derivative of $\hat{f}_x:T_x\mathcal{M}\rightarrow \mathbb{R}$ at $0_x$ in direction $p\in T_x\mathcal{M}$. Then the Riemannian subdifferential of $f$ at $x$ is defined as
	$$\partial f(x)=\partial\hat{f}_x(0_x).$$
\end{definition}

The following lemma shows some important properties of the Riemannian subdifferential, which are 
similar to those of the Clarke subdifferential in Hilbert space; see \cite[Thm. 2.2]{Hosseini2}.

\begin{lemma}\label{le2.1}
	Let $f:\mathcal{M}\rightarrow \mathbb{R}$ be a locally Lipschitz function, then the set $\partial f(x)$ is a nonempty, compact and convex subset of $T_x\mathcal{M}$, and $\|\xi\|\leq L$ for all $\xi\in \partial f(x)$, where $L$ is the Lipschitz constant near $x$.
\end{lemma}

The Riemannian $\varepsilon$-subdifferential and $\varepsilon$-subgradient of $f$ can be also defined; see  \cite{Hosseini2}.
\begin{definition}\label{def2.7}
	Let $\varepsilon\in\left(0,\frac{1}{2}\iota(x)\right)$, then the Riemannian $\varepsilon$-subdifferential of a locally Lipschitz function $f$ on a Riemannian manifold $\mathcal{M}$ at $x$ is defined as

	$$\partial_\varepsilon f(x):={\rm cl\,conv}\{\beta_\eta^{-1}\mathcal{T}_{x\leftarrow y}(\partial f(y)):\ y\in{\rm cl}B_R(x,\varepsilon)\,\, {\rm and}\,\, \eta=R_x^{-1}(y)\},$$
	where $\beta_\eta=\frac{\|\eta\|}{\|\mathcal{T}_{R_{\eta}}(\eta)\|}$.
	Every element of $\partial_\varepsilon f(x)$ is called a (Riemannian) $\varepsilon$-subgradient.
\end{definition}

\begin{remark}
   { The coefficient $\frac{1}{2}$ of $\iota(x)$ in Definition \ref{def2.7} is to guarantee that the notation $\mathcal{T}_{x\leftarrow y}(\partial f(y))$ is well defined on the boundary of $B_R(x,\varepsilon)$.}
\end{remark}

\begin{lemma}\label{le2.2}
	The set $\partial_\varepsilon f(x)$ is a nonempty, compact and convex  subset of $T_x\mathcal{M}$.
\end{lemma}

\begin{proof}
	From \cite[Thm. 2.15]{Hosseini2} we know that the set $\partial_\varepsilon f(x)$ is bounded. This together with Lemma \ref{le2.1} and Definition \ref{def2.7} shows the claim.
\end{proof}

\section{An implementable descent method for Riemannian nonsmooth multiobjective optimization}\label{sec3}

In this section, we present the details of our algorithm, which contains two procedures that help us to find a descent direction. We first introduce the definition of global weak Pareto optimal points and local weak Pareto optimal points for problem (\ref{pro1}) as follows; see \cite[Ch. 2]{Ehrgott}.
\begin{definition}\label{def3.1}
	{Let $x\in\mathcal{M}$. If there is no $y\in\mathcal{M}$ such that $f_i(y)< f_i(x)$ for all $i=1,\cdots ,m$, then we say that $x$ is a weak Pareto optimizer of problem (\ref{pro1}). If there exists some $\sigma>0$ such that $x$ is a (weak) Pareto optimizer on $\mathbb{B}(x,\sigma)\cap \mathcal{M}$, then we say that $x$ is a local (weak) Pareto optimizer of problem (\ref{pro1}).
 }
	
\end{definition}

It is clear that if $x$ is a Pareto optimal point of problem (\ref{pro1}), then it is a global weak Pareto optimal point of problem (\ref{pro1}), so it also must be a local weak Pareto optimal point of problem (\ref{pro1}). 
Next, we generalize the necessary condition for Pareto optimality in Euclidean space (see \cite{Makela2}) to the Riemannian setting.

\begin{theorem}\label{le3.1}\label{th3.1}
	Let $x\in\mathcal{M}$ be a local weak Pareto optimizer of problem (\ref{pro1}), then we have
	\begin{equation}\label{eq3.1}
		0_x\in {\rm conv}G(x),
	\end{equation}
	where $G(x)=\bigcup_{i=1}^m\partial f_i(x).$
\end{theorem}

\begin{proof}
	We first show that $\overline{G}(x)=\emptyset$, 
	where 
	$$
	\overline{G}(x)=\{d\in T_x\mathcal{M}|\,\,\langle d,\xi\rangle<0 \,\,{\rm for\,\,all\,\,}\xi\in G(x)\}.
	$$
	By Definition \ref{def3.1}, there exists a $\sigma>0$ such that for every $y\in \mathcal{M}\cap \mathbb{B}(x,\sigma)$ there is an index $i\in \{1,\cdots,m\}$ such that inequality $f_i(y)\geq f_i(x)$ holds. 
	Let $d\in T_x\mathcal{M}$ be arbitrary, then there exist sequences $\{d_k\}\subseteq T_x\mathcal{M}$ and $\{t_k\}$ such that $d_k\rightarrow d$ and $t_k\downarrow 0$. Set $\tilde{\varepsilon}=\min\{\frac{\sigma}{\kappa},\iota(x)\}$ and $\varepsilon\in(0,\tilde{\varepsilon})$. It is clear that there exists a constant $N>0$ such that $t_kd_k\in B(0_x,\varepsilon)$ for all $k>N$.
	By (\ref{eq2.1}), we have $R_x(t_kd_k)\in \mathcal{M}\cap \mathbb{B}(x,\sigma)$ for all $k>N$. Then for every $k>N$ there exists an index $i_k\in \{1,\cdots,m\}$ such that $f_{i_k}(R_x(t_kd_k))\geq f_{i_k}(x)$. Since $m$ is finite, there must be an index $i_0\in \{1,\cdots,m\}$ and subsequences $\{d_{k_j}\}\subset \{d_k\}$ and $\{t_{k_j}\}\subset\{t_k\}$ such that 
	$$f_{i_0}\left(R_x(t_{k_j}d_{k_j})\right)\geq f_{i_0}(x)=f_{i_0}(R_x(0_x)),$$
	for all $k_j>N$. 
	Denote $K=\{k_j|\,\,k_j>N\}$. Since $\hat{f}_{i_0x}=f_{i_0}\circ R_x$ is a locally Lipschiz function on $B(0_x,\varepsilon)$, by the mean value theorem (\cite[Thm. 2.3.7]{Clarke.F}), it follows that for all $\bar{k}\in K$, there exists a $\tilde{t}_{\bar{k}}\in (0,t_{\bar{k}})$ such that
	$$\hat{f}_{i_0x}(t_{\bar{k}}d_{\bar{k}})-\hat{f}_{i_0x}(0_x)\in \langle\partial\hat{f}_{i_0x}(\tilde{t}_{\bar{k}}d_{\bar{k}}),t_{\bar{k}}d_{\bar{k}}\rangle.$$
	Then from \cite[Prop. 2.1.2 (b)]{Clarke.F}, we obtain
	$$\hat{f}_{i_0x}^\circ(\tilde{t}_{\bar{k}}d_{\bar{k}};d_{\bar{k}})= \max_{\xi\in\partial\hat{f}_{i_0x}(\tilde{t}_{\bar{k}}d_{\bar{k}})}\langle\xi,d_{\bar{k}}\rangle\geq \frac{1}{t_{\bar{k}}}(\hat{f}_{i_0x}(t_{\bar{k}}d_{\bar{k}})-\hat{f}_{i_0x}(0_x))\geq 0.$$
	Thus, for all $\bar{k}\in K$ we have $\hat{f}_{i_0x}^\circ(\tilde{t}_{\bar{k}}d_{\bar{k}};d_{\bar{k}})\geq 0$. Since $d_{\bar{k}}\rightarrow d$ and $\tilde{t}_{\bar{k}}d_{\bar{k}}\rightarrow 0_x$, 
	from the upper semicontinuous of function $\hat{f}_{i_0x}^\circ$ (see \cite[Prop. 2.1.1 (b)]{Clarke.F}) and Definition \ref{def2.6}, we obtain
	$$f_{i_0}^\circ(x,d)=\hat{f}_{i_0x}^\circ(0_x;d)=\limsup_{\bar{k}\rightarrow \infty}\hat{f}_{i_0x}^\circ(\tilde{t}_{\bar{k}}d_{\bar{k}};d_{\bar{k}})\geq 0.$$
	By \cite[Thm. 2.2 (b)]{Hosseini2}, we have $f_{i_0}^\circ(x,d)=\max_{\xi\in\partial f_{i_0}(x)}\langle\xi,d\rangle$. Therefore, there exists a $\bar{\xi}\in\partial f_{i_0}(x)\subseteq G(x)$ such that $\langle\bar{\xi},d\rangle\geq 0$, which implies $d\notin \overline{G}(x)$, and thus $\overline{G}(x)=\emptyset$.
	
	Now, we show that $0_x\in {\rm conv}G(x)$. Note that $\overline{G}(x)=\emptyset$, then for any $d\in T_x\mathcal{M}$, there exists some $\xi_0\in G(x)\subseteq {\rm conv}G(x)$ such that 
	\begin{equation}\label{eq3.1.1}
		\langle d,\xi_0\rangle\geq 0.
	\end{equation}
	Suppose $0_x\notin {\rm conv}G(x)$. Since the sets ${\rm conv}G(x)$ and $\{0_x\}$ are closed convex sets, there exist $d\in T_x\mathcal{M}$ and $a\in\mathbb{R}$ such that $$0=\langle d,0_x\rangle\geq a \,\,{\rm and}\,\,\langle d,\xi\rangle<a \,\,{\rm for\,\,all}\,\, \xi\in{\rm conv}G(x)$$ 
	according to the separation theorem. The above relations imply that $\langle d,\xi\rangle<0$ for all $\xi\in{\rm conv}G(x)$, which is a contradiction with inequality (\ref{eq3.1.1}). Hence $0_x\in {\rm conv}G(x)$. 
\end{proof}

From Theorem \ref{th3.1} and the previous results, we know that $0_x\in {\rm conv}G(x)$ if $x$ is a Pareto optimizer of problem (\ref{pro1}).  Conversely, when the objective functions are strictly convex, the point 
satisfying (\ref{eq3.1}) is Pareto optimizer of problem (\ref{pro1}); see the lemma below.

\begin{lemma}
	Suppose that the objective functions of problem (\ref{pro1}) are all strictly convex on $\mathcal{M}$.\footnote{A function $f:\mathcal{M}\rightarrow \mathbb{R}$ is said to be convex if the composition $f\circ\gamma:[-\nu,\nu]\rightarrow \mathbb{R}$ is convex 
		for any geodesic segment $\gamma:[-\nu,\nu]\rightarrow \mathcal{M}$ with $\nu>0$; see \cite{Bento5}.  The function $f$ is said to be strictly convex if the composition $f\circ\gamma:[-\nu,\nu]\rightarrow \mathbb{R}$ is strictly convex.} Then every point satisfying (\ref{eq3.1}) is Pareto optimal point of problem (\ref{pro1}).
\end{lemma}

\begin{proof}
	By \cite[Lem. 1.3]{Attouch H}, we immediately obtain this result.
\end{proof}

The method proposed in this paper is a descent method based on the line search strategy.
In particular, for each iteration $k$, we hope to find a descent direction $g_k\in T_{x_k}\mathcal{M}$ and a stepsize $t_k>0$ such that $f_i(x_{k+1})<f_i(x_k)$ for all $i\in\{1,\cdots,m\}$, where $x_{k+1}=R_{x_k}(t_kg_k)$. Next, we will explain how to find such $g_k$. 


\begin{definition}\label{def3.2}
	For $\varepsilon\in \left(0,\frac{1}{2}\iota(x)\right)$, the $\varepsilon$-subdifferential
	of the objective vector $F(x)$ of problem  (\ref{pro1}) is defined as
	$$G_\varepsilon(x):={\rm conv}\left(\bigcup_{i=1}^m\partial_\varepsilon f_i(x)\right)\subset T_x{\mathcal{M}}.$$
\end{definition}

It is clear that $G_\varepsilon(x)$ is nonempty, convex and compact by Lemma \ref{le2.2}.

\begin{lemma}\label{le3.2}
	Let $\varepsilon\in \left(0,\frac{1}{2}\iota(x)\right)$.
	
	{ (i) If $x$ is a Pareto optimal point of problem (\ref{pro1})} , then
	$0_x\in G_\varepsilon(x)$.
	
	(ii) Let $x\in\mathcal{M}$ and
	\begin{equation}\label{eq3.2}
		\bar{g}:=\mathop{\rm argmin}\limits_{\xi\in-G_\varepsilon(x)}\|\xi\|.
	\end{equation}
	Then, either  $\bar{g}=0_x$ or $\bar{g}\neq 0_x$ and
	\begin{equation}\label{eq3.3}
		\langle\bar{g},\xi\rangle\leq-\|\bar{g}\|^2<0,\,\,\,\,\forall\xi\in G_\varepsilon(x).
	\end{equation}
\end{lemma}

\begin{proof}
	(i) It is obvious that $ {\rm conv}G(x)\subseteq G_\varepsilon(x)$, then combining with Theorem \ref{th3.1} we immediately obtain $0_x\in G_\varepsilon(x)$. (ii) Since the set $G_\varepsilon(x)$ is nonempty and compact, then there exists some $\bar{g}$ such that $\bar{g}=\mathop{\rm argmin}_{\xi\in-G_\varepsilon(x)}\|\xi\|$. If $\bar{g}\neq 0_x$, we have $-\bar{{g}}={\rm argmin}_{\xi\in G_\varepsilon(x)}\|\xi\|$ by (\ref{eq3.2}). Note that the set $G_\varepsilon(x)$ is convex, so we have the inequality $\langle\xi-(-\bar{g}),-(-\bar{g})\rangle\leq 0$, which implies (\ref{eq3.3}).
\end{proof}

By Lemma \ref{le3.2}, we still have the necessary optimality condition $0_x\in G_\varepsilon(x)$ when working with the $\varepsilon$-subdifferential instead of ${\rm conv}G(x)$. The following lemma shows that for each objective function $f_i$, there exists a common lower bound for a stepsize to guarantee descent when using the direction $\bar{g}$ defined by (\ref{eq3.2}) as a search direction. We extend the result of \cite[Lem. 3.2]{Gebken B} to the Riemannian setting as follows.

\begin{lemma}\label{le3.3}
	Let $\varepsilon\in \left(0,\frac{1}{2}\iota(x)\right)$ and $\bar{g}$ be the solution of (\ref{eq3.2}). Then
	$$f_i(R_x(t\bar{g}))\leq f_i(x)-t\|\bar{g}\|^2,\,\,\,\,\forall\,\, 0\leq t\leq\frac{\varepsilon}{\|\bar{g}\|},\,\,i\in\{1,\cdots,m\}.$$
\end{lemma}
\begin{proof}
	For all $t\in\left[0,\frac{\varepsilon}{\|\bar{g}\|}\right]$, by Lebourg's mean value theorem \cite[Thm. 3.3]{Hosseini3}, there exist $\theta\in(0,1)$ and $\xi\in\partial f_i(R_x(\theta t\bar{g}))$ such that
	$$f_i(R_x(t\bar{g}))-f_i(x)=\langle\xi,{\rm D}R_x(\theta t\bar{g})[t\bar{g}]\rangle,\,\,i\in\{1,\cdots,m\}.$$
	It is clear that $\|\theta t\bar{g}\|<\varepsilon$. Combining (\ref{inner}) and the locking condition (\ref{locking}) of the vector transport, we have that
	$$f_i(R_x(t\bar{g}))-f_i(x) = \frac{t}{\beta_{\theta t\bar{g}}}\langle\mathcal{T}_{x\leftarrow R_x(\theta t\bar{g})}(\xi),\bar{g}\rangle,\,\,i\in\{1,\cdots,m\}.$$
	Since $\|\theta t\bar{g}\|<\varepsilon$, it follows that $\frac{1}{\beta_{\theta t\bar{g}}}\mathcal{T}_{x\leftarrow R_x(\theta t\bar{g})}(\xi)\in\partial_\varepsilon f_i(x)\subset G_\varepsilon(x)$, then from (\ref{eq3.3}) we obtain
	$$f_i(R_x(t\bar{g}))-f_i(x)\leq-t\|\bar{g}\|^2,\,\,i\in\{1,\cdots,m\}.$$
	This completes the proof.
\end{proof}

Lemma \ref{le3.3} states that $\bar{g}$ is a descent direction of $f_i$ for every { $i\in\{1,\cdots,m\}$}.
However, it is not easy to compute $\bar{g}$ in practice, since the set $G_\varepsilon(x)$ is usually unknown. A natural idea is to approximate $G_\varepsilon(x)$ by the convex hull of a certain set $W$,
which is expected to have at least two properties: 
(i) it is much easier to compute than $G_\varepsilon(x)$; 
(ii) $\tilde{g}={\rm argmin}_{\xi\in -{\rm conv}W}\|\xi\|$ can be used instead of $\bar{g}$ as an approximate descent direction of $f_i$ for every { $i\in\{1,\cdots,m\}$}.

Now, we present the details of our algorithm (Algorithm \ref{al1}).

\LinesNumbered
\begin{algorithm}\label{al1}
	\setcounter{AlgoLine}{-1}
	\caption{A descent method for Riemannian nonsmooth multiobjective optimization}
	\normalsize
	Select an initial point $x_0\in\mathcal{M}$,   $\varepsilon\in\left(0,\frac{1}{2}\iota(\mathcal{M})\right)$, tolerance $\delta>0$, and Armijo parameters $c\in(0,1)$, $\alpha>1$, $t_0>0$. Let $k=0$.
	
	Compute an acceptable descent direction: $g_k=\mathcal{P}_{\rm dd}(x_k,\varepsilon,\delta,c)$, where $\mathcal{P}_{\rm dd}$ is a procedure given below.
	
	If $\|g_k\|\leq\delta$, then STOP.
	
	Find the smallest integer $\ell\in\left\{0,1,\cdots,\left[\frac{{\rm ln}(t_0\|g_k\|)-{\rm ln}\varepsilon}{{\rm ln}\alpha}\right]\right\}$ satisfying
	\begin{equation*}\label{eq3.4}
		f_i(R_{x_k}(\alpha^{-\ell}t_0g_k))\leq f_i(x_k)-\alpha^{-\ell}t_0c\|g_k\|^2,\,\,\ \  i\in\{1,\cdots,m\}.
	\end{equation*}
	If such an $\ell$ exists, set $t_k=\alpha^{-\ell}t_0$. Otherwise, set $t_k=\frac{\varepsilon}{\|g_k\|}$.
	
	Set $x_{k+1}=R_{x_k}(t_kg_k),\,\,k=k+1$ and go to Step 1.
\end{algorithm}

\begin{remark}
	In Step 1 of Algorithm \ref{al1}, the aim of the inner procedure $\mathcal{P}_{\rm dd}$
	is to find an acceptable descent direction of $f_i$ for every { $i\in\{1,\cdots,m\}$}, 
	which uses the substitute ${\rm conv}W$ instead of $G_\varepsilon(x)$.	
	In Step 3, the symbol $[\,\cdot\,]$ denotes the largest integer that does not exceed $\cdot$.
	In what follows, we will show that $\frac{\varepsilon}{\|g_k\|}$ is a common descent stepsize for all objective functions when using $g_k$ as the search direction. 
	The line search strategy of Step 3 means that if there is a longer stepsize $\alpha^{-\ell}t_0$ than $\frac{\varepsilon}{\|g_k\|}$, then we use $\alpha^{-\ell}t_0$ as the stepsize. Otherwise we use the latter.
\end{remark}

Next, we describe how we can obtain a good approximation of $G_\varepsilon(x)$ without requiring full knowledge of the $\varepsilon$-subdifferential. Let $W=\{\xi_1,\cdots,\xi_r\}\subseteq G_\varepsilon(x)$ and
\begin{equation}\label{eq3.5}
	\tilde{g}:=\mathop{\rm argmin}\limits_{g\in-{\rm conv}W}\|g\|.
\end{equation}
If $\tilde{{g}}\neq 0_x$, then set $c\in(0, 1)$ and check the following inequality:
\begin{equation}\label{eq3.6}
	f_i\left(R_x\left(\frac{\varepsilon}{\|\tilde g\|}\tilde{g}\right)\right)\leq f_i(x)-c\varepsilon\|\tilde{g}\|,\,\,\,\,\forall\,\, i\in\{1,\cdots,m\}.
\end{equation}
If (\ref{eq3.6}) holds, then we can say ${\rm conv}W$ is an acceptable approximation for $G_\varepsilon(x)$, and $\tilde{g}$ is an  acceptable descent direction. Otherwise, the set $I\subseteq\{1,\cdots,m\}$ consists of the indices for which (\ref{eq3.6}) is not satisfied is nonempty,
then we hope to find a new $\varepsilon$-subgradient $\xi'\in G_\varepsilon(x)$ such that $W\cup\xi'$ yields a better approximation to $G_\varepsilon(x)$. The following lemma can help us to find such an $\varepsilon$-subgradient.

\begin{lemma}\label{le3.4}
	Let $c\in(0, 1)$, $W=\{\xi_1,\cdots,\xi_r\}\subseteq G_\varepsilon(x)$ and $\tilde{g}\neq 0_x$ be the solution of (\ref{eq3.5}). If
	\begin{equation}\label{eq3.6.1}
		f_j\left(R_x\left(\frac{\varepsilon}{\|\tilde{g}\|}\tilde{g}\right)\right)> f_j(x)-c\varepsilon\|\tilde{g}\|
	\end{equation}
	for some $j\in\{1,\cdots,m\}$, then there exist some $t'\in\left[0,\frac{\varepsilon}{\|\tilde{\rm g}\|}\right]$ and $\xi'\in\partial f_j(R_x(t'\tilde{g}))$ such that
	\begin{equation}\label{eq3.7}
		\langle\beta^{-1}_{t'\tilde{g}}\mathcal{T}_{x\leftarrow R_x(t'\tilde{g})}\xi',\tilde{g}\rangle>-c\|\tilde{g}\|^2,
	\end{equation}
	and 
	\begin{equation}\label{eq3.7b}
		\xi=\beta^{-1}_{t'\tilde{g}}\mathcal{T}_{x\leftarrow R_x(t'\tilde{g})}\xi'\in G_\varepsilon(x)\setminus {\rm conv}W.
	\end{equation}
	
\end{lemma}

\begin{proof}
	Suppose for all $t'\in\left[0,\frac{\varepsilon}{\|\tilde{g}\|}\right]$ and $\xi'\in\partial f_j(R_x(t'\tilde{g}))$ we have
	\begin{equation}\label{eq3.8}
		\langle\beta^{-1}_{t'\tilde{g}}\mathcal{T}_{x\leftarrow R_x(t'\tilde{g})}\xi',\tilde{g}\rangle\leq-c\|\tilde{g}\|^2.
	\end{equation}
	Next we show that it is impossible.  In fact, by Lebourg's mean value theorem, there exist $\theta\in(0,1)$ and $\tilde{\xi}\in\partial f_j\left(R_x\left(\theta\frac{\varepsilon}{\|\tilde{g}\|}\tilde{g}\right)\right)$ such that
	$$f_j\left(R_x\left(\frac{\varepsilon}{\|\tilde{g}\|}\tilde{g}\right)\right)-f_j(x)=\left\langle\tilde{\xi},{\rm D}R_x\left(\theta\frac{\varepsilon}{\|\tilde{g}\|}\tilde{g}\right)\left[\frac{\varepsilon}{\|\tilde{g}\|}\tilde{g}\right]\right\rangle.$$
	Note that $\|\theta\frac{\varepsilon}{\|\tilde{g}\|}\tilde{g}\|<\varepsilon$, then using (\ref{inner}) and the locking condition (\ref{locking}) of the vector transport, we obtain
	$$f_j\left(R_x\left(\frac{\varepsilon}{\|\tilde{g}\|}\tilde{g}\right)\right)-f_j(x)=\frac{\varepsilon}{\|\tilde{g}\|\beta_{\theta\frac{\varepsilon}{\|\tilde{g}\|}\tilde{g}}}\left\langle\mathcal{T}_{x\leftarrow R_x\left(\theta\frac{\varepsilon}{\|\tilde{g}\|}\tilde{g}\right)}(\tilde{\xi}),\tilde{{g}}\right\rangle.$$
	This together with $\theta\in(0,1)$ and (\ref{eq3.8}) shows that 
	$$f_j\left(R_x\left(\frac{\varepsilon}{\|\tilde{g}\|}\tilde{g}\right)\right)-f_j(x)\leq \frac{\varepsilon}{\|\tilde{g}\|}(-c\|\tilde{g}\|^2)= -c\varepsilon\|\tilde{g}\|,$$
	which is a contradiction with (\ref{eq3.6.1}), and therefore (\ref{eq3.7}) holds.
	
	Finally, we prove (\ref{eq3.7b}). 
	Note that $t'\in\left[0,\frac{\varepsilon}{\|\tilde{g}\|}\right]$, thus $\|t'\tilde{g}\|\leq\varepsilon$, which implies that
	$\xi=\beta^{-1}_{t'\tilde{g}}\mathcal{T}_{x\leftarrow R_x(t'\tilde{g})}\xi'\in G_\varepsilon(x)$. If $\xi\in {\rm conv}W$, we have $\langle\xi,\tilde{g}\rangle\leq-\|\tilde{g}\|^2<-c\|\tilde{g}\|^2$, which is a contradiction with (\ref{eq3.7}). Thus (\ref{eq3.7b}) holds.
\end{proof}

Lemma \ref{le3.4} above only implies that there exist $t'$ and $\xi'$ 
satisfying (\ref{eq3.7}) without showing a way how to obtain them.
Now, we present a procedure ($\mathcal{P}_{\rm ns}$) which can help us to compute such $t'$ and $\xi'$ in practice. For simplicity, denote
\begin{equation*}\label{eq3.9}
	h_j(t):=f_j(R_x(t\tilde{g}))-f_j(x)+ct\|\tilde{g}\|^2
\end{equation*}
for $j\in I$, i.e., $f_j$ is not satisfied with (\ref{eq3.6}). 

\vskip 0.3cm

\renewcommand{\algorithmcfname}{Procedure}
\renewcommand{\thealgocf}{}
\LinesNumbered
\begin{algorithm}[H]
	\setcounter{AlgoLine}{-1}
	\caption{Find a new $\varepsilon$-subgradient: $(t,\tilde{\xi}^j_s)=\mathcal{P}_{\rm ns}(j,x_k,\tilde{g}_s,\varepsilon,c)$}
	
	\normalsize
	
	Set $a=0,b=\frac{\varepsilon}{\|\tilde{g}_s\|}$ and $t=\frac{a+b}{2}$.
	
	Compute $\tilde{\xi}^j_s\in\partial f_j(R_{x_k}(t\tilde{g}_s))$.
	
	If $\langle\beta^{-1}_{t\tilde{g}_s}\mathcal{T}_{x_k\leftarrow R_{x_k}(t\tilde{g}_s)}\tilde{\xi}^j_s,\tilde{g}_s\rangle>-c\|\tilde{g}_s\|^2$ then STOP.
	
	If $h_j(b)>h_j(t)$ set $a=t$. Else set $b=t$.
	
	Set $t=\frac{a+b}{2}$ and go to Step 1.
\end{algorithm}

Based on this procedure,
it is possible to construct another procedure that can compute an acceptable descent direction of $f_i$ for {$i\in\{1,\cdots,m\}$}, namely procedure $\mathcal{P}_{\rm dd}$ used in Step 1 of Algorithm \ref{al1}.

\renewcommand{\algorithmcfname}{Procedure}
\renewcommand{\thealgocf}{}
\LinesNumbered
\begin{algorithm}
	\setcounter{AlgoLine}{-1}
	\caption{Compute a descent direction: $g_k=\mathcal{P}_{\rm dd}(x_k,\varepsilon,\delta,c)$}
	
	\normalsize
	
	Compute $\xi^i\in\partial_\varepsilon f_i(x_k)$ for all $i\in\{1,\cdots,m\}$. Set $W_0=\{\xi^1,\cdots,\xi^m\}$ and $s=0$.
	
	Compute $\tilde{g}_s=\mathop{\rm argmin}\limits_{g\in-{\rm conv}W_s}\|g\|$.
	
	If $\|\tilde{g}_s\|\leq \delta$, set $g_k=\tilde{g}_s$ and STOP.
	
	Find all indices for which there is no sufficient decrease:
	$$I_s=\left\{j\in\{1,\cdots,m\}:\ f_j\left(R_{x_k}\left(\frac{\varepsilon}{\|\tilde{g}_s\|}\tilde{g}_s\right)\right)>f_j(x_k)-c\varepsilon\|\tilde{g}_s\|\right\}.$$
	If $I_s=\emptyset$, set $g_k=\tilde{g}_s$, then STOP.
	
	For each $j\in I_s$, compute $(t,\tilde{\xi}^j_s)=\mathcal{P}_{\rm ns}(j,x_k,\tilde{g}_s,\varepsilon,c)$, and set $\xi_s^j=\beta_{t\tilde{g}_s}^{-1}\mathcal{T}_{x_k\leftarrow R_{x_k}(t\tilde{g}_s)}\tilde{\xi}_s^j$.
	
	Set $W_{s+1}=W_s\cup\{\xi_s^j:\ j\in I_s\},\,\,s=s+1$ and go to Step 1.
\end{algorithm}

\section{Convergence analysis}\label{sec4}
In this section, we show that the Algorithm \ref{al1} is well defined at first. That is, the Procedure $\mathcal{P}_{\rm ns}$ and $\mathcal{P}_{\rm dd}$ can be terminated in a finite number of iterations, respectively. And then we establish the convergence of Algorithm \ref{al1}.

\begin{lemma}\label{th3.2}
	For the current point $x_k$, let $j\in I$ and $\{t_i\}$ be the sequence generated by the Procedure $\mathcal{P}_{\rm ns}$.
	
	(i) If $\{t_i\}$ is finite, then some $\xi'$ was found to satisfy (\ref{eq3.7}).
	
	(ii) If $\{t_i\}$ is infinite, then it converges to some $\bar{t}\in\left[0,\frac{\varepsilon}{\|\tilde{g}_s\|}\right]$ such that either there is some $\xi'\in\partial f_j(R_{x_k}(\bar{t}\tilde{g}_s))$ satisfying (\ref{eq3.7}) or $0\in\partial h_j(\bar{t})$.
\end{lemma}
\begin{proof}
	(i) If $\{t_i\}$ is finite, by construction, the procedure $\mathcal{P}_{\rm ns}$ must have stopped in Step 2, then some $\xi'=\tilde{\xi}^j_s$ was found to satisfy (\ref{eq3.7}).
	
	(ii) If $\{t_i\}$ is infinite, it is clear that $\{t_i\}$ must be convergent to some $\bar{t}\in\left[0,\frac{\varepsilon}{\|\tilde{g}_s\|}\right]$. Additionally, we have $h_j(0)=0$ and $h_j\left(\frac{\varepsilon}{\|\tilde{g}_s\|}\right)>0$ since (\ref{eq3.6}) is violated for the index $j$. Let $\{a_i\}$ and $\{b_i\}$ be the sequences corresponding to $a$ and $b$ in Procedure $\mathcal{P}_{\rm ns}$. 
	
	We first show that $\bar{t}\neq 0$. Suppose by contradiction that $\bar{t}=0$. By the construction of $\mathcal{P}_{\rm ns}$, we have $h_j(t_i)\geq h_j(b_i)$ for all $i\in\mathbb{N}$. Then
	$$h_j(t_i)\geq h_j(b_i)=h_j(t_{i-1})\geq h_j(b_{i-1})=\cdots=h_j(t_1)\geq h_j(b_1)=h_j\left(\frac{\varepsilon}{\|\tilde{g}_s\|}\right)>0.$$
	Due to the continuity of $h_j$, we obtain $h_j(0)=\lim_{i\rightarrow\infty}h_j(t_i)\geq h_j\left(\frac{\varepsilon}{\|\tilde{g}_s\|}\right)>0$, which is a contradiction. 
	
	So we must have $\bar{t}>0$. Furthermore, it is clear that $h_j(b_i)>h_j(a_i)$ for all $i\in\mathbb{N}$ by the construction of $\mathcal{P}_{\rm ns}$. Since the function $h_j$ is locally Lipschiz continuous on $\left[0,\frac{\varepsilon}{\|\tilde{g}_s\|}\right]$, by the mean value theorem there is some $r_i\in[a_i,b_i]$ such that
	$$0<h_j(b_i)-h_j(a_i)\in(b_i-a_i)\partial h_j(r_i).$$
	It is obvious that $\lim_{i\rightarrow\infty}r_i=\bar{t}$ and $a_i<b_i$, thus we have $\partial h_j(r_i)\cap\mathbb{R}_+\neq\emptyset$ for all $i\in\mathbb{N}$. Due to the upper semicontinuity of $\partial h_j$, there must be some $v\in\partial h_j(\bar{t})$ with $v\geq 0$. By \cite[Prop. 3.1]{Hosseini3}, we obtain
	\begin{equation}\label{eq3.10}
		0\leq v\in\partial h_j(\bar{t})\subseteq\langle\partial f_j(R_{x_k}(\bar{t}{\tilde{g}}_s)),{\rm D}R_{x_k}(\bar{t}\tilde{g}_s)[\tilde{g}_s]\rangle+c\|\tilde{g}_s\|^2.
	\end{equation}
	Thus, if there is some $\xi_0\in\partial f_j(R_{x_k}(\bar{t}\tilde{g}_s))$ with
	$0<\langle\xi_0,{\rm D}R_{x_k}(\bar{t}\tilde{g}_s)[\tilde{g}_s]\rangle+c\|\tilde{g}_s\|^2$, 
	i.e.
	$\langle\xi_0,{\rm D}R_{x_k}(\bar{t}\tilde{g}_s)[\tilde{g}_s]\rangle>-c\|\tilde{g}_s\|^2,$
	then using (\ref{inner}) and the locking condition (\ref{locking}) of vector transport, we have that
	$$\langle\beta^{-1}_{\bar{t}\tilde{g}_s}\mathcal{T}_{x_k\leftarrow R_{x_k}(\bar{t}\tilde{g}_s)}\xi_0,\tilde{g}_s\rangle>-c\|\tilde{g}_s\|^2,$$
	which shows that $\xi'=\xi_0\in\partial f_j(R_{x_k}(\bar{t}\tilde{g}_s))$ satisfies (\ref{eq3.7}). Otherwise, we obtain
	$$\langle\xi,{\rm D}R_{x_k}(\bar{t}\tilde{g}_s)[\tilde{g}_s]\rangle+c\|\tilde{g}_s\|^2\leq 0, \ \ \ \forall \xi\in\partial f_j(R_{x_k}(\bar{t}\tilde{g}_s)).
	$$
	This along (\ref{eq3.10}) implies $0=v\in\partial h_i(\bar{t})$.
\end{proof}

In fact, the Procedure ($\mathcal{P}_{\rm ns}$) will stop after finitely many iterations in practice; see \cite[Remark 3.1]{Gebken B}.

\begin{lemma}\label{th3.3}
	The Procedure $\mathcal{P}_{\rm dd}$ terminates in a finite number of iterations. In addition, let  $\hat{g}$ be the last element of $\{\tilde{g}_s\}$, then either $\|\hat{g}\|\leq\delta$ or $\hat{g}$ is an acceptable descent direction, that is
	$$f_i\left(R_{x_k}\left(\frac{\varepsilon}{\|\hat{g}\|}\hat{g}\right)\right)\leq f_i(x_k)-c\varepsilon\|\hat{g}\|, \ \ \ \forall i\in\{1,\cdots,m\}.$$
\end{lemma}
\begin{proof}
	Suppose by contradiction that $\{\tilde{g}_s\}$ is an infinite sequence. Let $s\geq1$ and $j\in I_{s-1}$. By the construction of $\mathcal{P}_{\rm dd}$, it follows that $\xi_{s-1}^j\in W_s$ and $-\tilde{g}_{s-1}\in {\rm conv}W_{s-1}\subseteq {\rm conv}W_s$. Since $\tilde{g}_s={\rm argmin}_{g\in-{\rm conv}W_s}\|g\|$, for all $\lambda\in(0,1)$, we have
	\begin{eqnarray}\label{eq3.11}
		{ \|\tilde{g}_s\|^2}
		& {<} & {\|-\tilde{g}_{s-1}+\lambda(\xi_{s-1}^j+\tilde{g}_{s-1})\|^2} \nonumber\\
		& = &\|\tilde{g}_{s-1}\|^2-2\lambda\langle\tilde{g}_{s-1},\xi_{s-1}^j\rangle-2\lambda\|\tilde{g}_{s-1}\|^2+\lambda^2\|\xi_{s-1}^j+\tilde{g}_{s-1}\|^2.
	\end{eqnarray}
	Note that $j\in I_{s-1}$, then by Step 4 of $\mathcal{P}_{\rm dd}$ and $\mathcal{P}_{\rm ns}$, we obtain
	\begin{equation}\label{eq3.12}
		\langle\tilde{g}_{s-1},\xi_{s-1}^j\rangle>-c\|\tilde{g}_{s-1}\|^2.
	\end{equation}
	Additionally, since $G_\varepsilon(x_k)$ is a compact subset of $T_{x_k}\mathcal{M}$, there is a constant $C>0$ such that $\|\xi\|\leq C$ for all $\xi\in G_\varepsilon(x)$. Thus
	\begin{equation}\label{eq3.13}
		\|\xi_{s-1}^j+\tilde{g}_{s-1}\|\leq 2C.
	\end{equation}
	Combining (\ref{eq3.11}) with (\ref{eq3.12}) and (\ref{eq3.13}), we have
	\begin{eqnarray*}
		\|\tilde{g}_s\|^2
		& < &\|\tilde{g}_{s-1}\|^2+2\lambda c\|\tilde{g}_{s-1}\|^2-2\lambda\|\tilde{g}_{s-1}\|^2+4\lambda^2C^2\\\nonumber
		& = &\|\tilde{g}_{s-1}\|^2-2\lambda(1-c)\|\tilde{g}_{s-1}\|^2+4\lambda^2C^2.\nonumber
	\end{eqnarray*}
	Let $\lambda=\frac{1-c}{4C^2}\|\tilde{g}_{s-1}\|^2$, then it follows from $c\in(0,1)$ and $\|g_{k-1}\|\leq C$ that $\lambda\in(0,1)$. Therefore
	\begin{eqnarray*}
		\|\tilde{g}_s\|^2
		& < &\|\tilde{g}_{s-1}\|^2-2\frac{(1-c)^2}{4C^2}\|\tilde{g}_{s-1}\|^4+\frac{(1-c)^2}{4C^2}\|\tilde{g}_{s-1}\|^4\\\nonumber
		& = &\left(1-\frac{(1-c)^2}{4C^2}\|\tilde{g}_{s-1}\|^2\right)\|\tilde{g}_{s-1}\|^2.\nonumber
	\end{eqnarray*}
	Since the $\mathcal{P}_{\rm dd}$ does not terminate, it holds $C\geq\|\tilde{g}_{s-1}\|>\delta$. Thus
	$$\|\tilde{g}_s\|^2<\left(1-\left(\frac{1-c}{2C}\delta\right)^2\right)\|\tilde{g}_{s-1}\|^2.$$
	Set $\tau=1-(\frac{1-c}{2C}\delta)^2\in(0,1)$. By recursion, we obtain
	$$\|\tilde{g}_s\|^2<\tau\|\tilde{g}_{s-1}\|^2<\tau^2\|\tilde{g}_{s-2}\|^2<\cdots<\tau^{s-1}\|\tilde{g}_1\|^2\leq \tau^{s-1}C^2.$$
	This shows that $\|\tilde{g}_s\|\leq\delta$ for sufficiently large $s$, which is a contradiction.
\end{proof}

Lemma \ref{th3.2} show that the Procedure $\mathcal{P}_{\rm dd}$ terminates in a finite number of iterations, and then an acceptable descent direction is produced. And thus the Algorithm \ref{al1} is well defined.

In the following, we extend the conception of $(\varepsilon,\delta)$-critical from $\mathbb{R}^{n}$ (see \cite[Def. 3.2]{Gebken B}) to Riemannian manifolds. Then under the assumption of at least one objective function of problem (\ref{pro1}) is bounded below, we show that the sequence $\{x_k\}$ generated by Algorithm \ref{al1} is finite with the last element being $(\varepsilon,\delta)$-critical.

\begin{definition}\label{def4.1}
	Let $x\in\mathcal{M}$, $\varepsilon\in\left(0,\frac{1}{2}\iota(\mathcal{M})\right)$ and $\delta>0$. Then, $x$ is called $(\varepsilon,\delta)$-critical, if
	$$\min_{g\in-G_\varepsilon(x)}\|g\|\leq\delta.$$
\end{definition}

Clearly, if a point $x\in\mathcal{M}$ satisfies (\ref{eq3.1}), then it is an $(\varepsilon,\delta)$-critical point, but the converse is not necessarily true.


\begin{theorem}\label{th4.1}
	Assume that at least one objective function of problem (\ref{pro1}) is bounded below. Let $\{x_k\}$ be the sequence generated by Algorithm \ref{al1}. Then $\{x_k\}$ is finite with the last element being $(\varepsilon,\delta)$-critical.
\end{theorem}
\begin{proof}
	Suppose by contradiction that $\{x_k\}$ is infinite. Then, we have $\|g_k\|>\delta$ for all $k\in\mathbb{N}$. 
	If $t_k=\alpha^{-\ell}t_0$ in Step 3 of Algorithm \ref{al1}, then we have  $\alpha^{-\ell}t_0\geq\frac{\varepsilon}{\|g_k\|}$. This together with (\ref{eq3.4}) shows that, for all $i\in\{1,\cdots,m\}$,
	\begin{equation}\label{temp3}
		\begin{aligned}
			f_i(R_{x_k}(t_kg_k))-f_i(x_k)
			& =  f_i(R_{x_k}(\alpha^{-\ell}t_0g_k))-f_i(x_k)\\
			& \leq  -\alpha^{-\ell}t_0c\|g_k\|^2\\
			& \leq  -c\varepsilon\|g_k\|\\
			& < -c\varepsilon\delta.
		\end{aligned}
	\end{equation}
	Conversely, if $t_k=\frac{\varepsilon}{\|g_k\|}$,
	we have $\frac{\varepsilon}{\|g_k\|}\geq\alpha^{-\ell}t_0$, 
	then from Lemma \ref{th3.3}, the last inequality in \eqref{temp3} can be also obtained. In summary, we can conclude that $\{f_i(x_k)\}$ is unbounded below for each $i\in\{1,\cdots,m\}$, which is a contradiction.Thus the sequence $\{x_k\}$ is finite.
	
	Let $x_*$ and $g_*$ be the last elements of $\{x_k\}$ and $\{g_k\}$, respectively. Since the algorithm stopped, we must have $\|g_*\|\leq\delta$ by Step 2 of Algorithm \ref{al1}. On the other hand, by the construction of the procedures $\mathcal{P}_{\rm dd}$ and $\mathcal{P}_{\rm ns}$, there is a set $W_*\subseteq G_\varepsilon(x_*)$ such that
	$g_*={\rm argmin}_{g\in-{\rm conv}W_*}\|g\|$. Thus
	$$\min_{g\in-G_\varepsilon(x_*)}\|g\|\leq\min_{g\in-{\rm conv}(W_*)}\|g\|=\|g_*\|\leq\delta,$$
	which completes the proof.
\end{proof}

\section{Numerical results}\label{sec5}
In this section, we will present numerical results of several examples for our method. Most of the objective functions of these examples are of the classic optimization problems on Riemannian manifolds. 
{
	It seems that our method is not suitable to be compared with the existing Riemannian nonsmooth multiobjective optimization methods \cite{Bento4} and \cite{Bento3}.
In fact, the subgradient method \cite{Bento4} 
assumes that the objective vector function is convex on $\mathcal{M}$ and requires to calculate the geodesics. However, although the geodesics of the examples we tested below have closed forms, the subgradient method may still not be applicable since these examples may not be convex on the manifolds. Note that it is usually not easy to verify whether a function on a manifold is convex, but we can verify that Example \ref{example1} is nonconvex when the sphere $S^2$ inherits the Euclidean metric of $\mathbb{R}^3$.}
In addition, the proximal point method \cite{Bento3} requires the constraint manifold to be a Hadamard manifold, which 
clearly limits its applications, since the Stiefel manifold used in our test below is not a Hadamard manifold. Based on the observations above,
in what follows we only report the numerical results produced by Algorithm \ref{al1} without comparisons.

\begin{example}\label{example1}
	We first consider a simple problem. Let $m=2$ in problem (\ref{pro1}), and set
	$\mathcal{M}=S^2$ which is the Euclidean unit sphere in $\mathbb{R}^3$, 
	$f_1(x)=\max(0.5x_1+x_3,0.3x_2+1.5x_3)$ and  $f_2(x)=|x_1-0.5|+x_2+x_3$. 	
\end{example}

\begin{example}\label{example2} 
	Recently, many researchers are interested in the geometric median on a Riemannian manifold $\mathcal{M}$ (see\cite{Hoseini,Fletcher}). Let $y_1,\cdots,y_q\in\mathcal{M}$ be some given points,  $w=(w_1,\cdots,w_q)^T\in\mathbb{R}^q_+$ and $\sum_{j=1}^{q}w_j=1$ be the corresponding weights. This problem is to minimize $\sum_{j=1}^{q}w_j{\rm dist}(x,y_j)$ on $\mathcal{M}$. Now, we consider the multiobjective setting and set $\mathcal{M}=S^{p-1}$ which is the Euclidean unit sphere in $\mathbb{R}^p$ and 
	$f_i(x)=\sum_{j=1}^{q^i}w_j^i{\rm dist}(x,y_j^i)$, where $y_1^i,\cdots,y_{q^i}^i\in\mathcal{M}$, $w^i\in\mathbb{R}^{q^i}_+$ with $\sum_{j=1}^{q^i}w_j^i=1$ for all $i\in \{1,\cdots,m\}$.
\end{example}


\begin{example}\label{example3}
	Eigenvalue problems are ubiquitous in scientific research and practical applications, such as physical science and engineering design, etc. 
	Let $A$ be a real symmetric matrix, the eigenvalue problem can be transformed into
	a Rayleigh quotient problem whose objective function is $\frac{x^TAx}{x^Tx}$. This problem can be further viewed as an optimization problem on a sphere to minimize $x^TAx$; see \cite{P. A. Absil}.
	Also, we consider the multiobjective setting, set $\mathcal{M}=S^{p-1}$ and 
	$f_i(x)=x^TA_ix$, where $A_i$ is a real symmetric matrix for each $i\in\{1,\cdots,m\}$.
\end{example}

\begin{example}\label{example4}
	The $l_1$-regularized least squares problem (named as Lasso) was proposed in \cite{Tibshirani R}, which has been used heavily in machine learning and basis pursuit denoising, etc. 
	The cost function of this problem is $\frac{1}{2}\|Ax-b\|^2+\lambda\|x\|_1$,
	where $A\in \mathbb{R}^{n\times p}$, $b\in\mathbb{R}^n$ and $\lambda>0$. Here, we restrict $x$ to the unit sphere $S^{p-1}$ and consider the objective functions $f_i(x)=\frac{1}{2}\|A_ix-b_i\|^2+\lambda_i\|x\|_1$, where $A_i\in\mathbb{R}^{n\times p}$, $b_i\in\mathbb{R}^n$ and $\lambda_i>0$ for all $i\in\{1,\cdots,m\}$.
\end{example}

On sphere $S^{p-1}$, the Riemannian metric is inherited from the ambient space $\mathbb{R}^p$, and the Riemannian distance ${\rm dist}(x,y)={\rm arccos}\langle x,y\rangle$. Moreover, for all instances, the exponential map and the parallel transport are employed as a retraction and vector transport, respectively. More precisely, the retraction is as follows
$$R_x(\xi):={\rm exp}_x(\xi)={\rm cos}(\|\xi\|)x+{\rm sin}(\|\xi\|)\frac{\xi}{\|\xi\|},$$
\noindent where $\xi\in T_xS^{p-1}$. The vector transport associated with $R$ is given by
$$\mathcal{T}_{x\rightarrow\gamma(t)}:=\left(I_p+({\rm cos}(\|\dot{\gamma}(0)t\|)-1)uu^T-{\rm sin}(\|\dot{\gamma}(0)t\|xu^T\right)\xi,$$
\noindent where $\gamma$ is a geodesic on $S^{p-1}$ with $\gamma(0)=x$ and $u=\frac{\dot{\gamma}(0)}{\|\dot{\gamma}(0)\|}$. Note that $\mathcal{T}_{x\leftarrow y}(\xi_y)=\mathcal{T}_{y\rightarrow\sigma(1)}$, where $\sigma(t)={\exp_y(tv)}$ denotes the geodesic connecting $y$ to $x$, and $v$ can be computed by
$v={\rm dist}(x,y)\frac{(I-xx^T)(y-x)}{\|(I-xx^T)(y-x)\|}.$
Therefore
$$\mathcal{T}_{x\leftarrow y}(\xi_y)=\left(I_p+({\rm cos}(\|v\|)-1)uu^T-{\rm sin}(\|v\|)yu^T\right)\xi_y\,\ {\rm with} \,\ u=v/\|v\|,$$
which is well defined for all $y\neq \pm x$; see \cite{Hosseini1}.

In order to obtain the element with the smallest norm in the set $W_s$ at Step 1 of Procedure $\mathcal{P}_{\rm dd}$, we require to choose the basis of the tangent space $T_{x^k}\mathcal{M}$ for all $x^k$. Particularly, Huang \cite{Huang W1} provided a way to select the orthogonal basis of $T_{x^k}\mathcal{M}$ for Stiefel manifold. Then we assume that such basis are given by $\{E_{k,1},\cdots,E_{k,d}\}$. Let $W_s=\{\xi_1,\cdots,\xi_r\}\subset G_\varepsilon(x^k)$ and $\xi_i=\sum_{j=1}^{d}\xi_i^jE_{k,j}$, and we solve the following quadratic programming:
$$\min_{\lambda\in\mathbb{R}^r}~~\lambda^T\widehat{W}^T_s\widehat{W}_s\lambda\,\, ~~{\rm such\,\,that}~~\,\,\sum_{j=1}^{r}\lambda_j=1\,\,{\rm and}\,\,\lambda\geq 0,$$
where $\widehat{W}_s=[\hat{\xi}_1,\cdots,\hat{\xi}_r]$ is a matrix with {
	 $\hat{\xi}_i=(\xi_i^1,\cdots,\xi_i^d)^T\in\mathbb{R}^d$. Let $\bar{\lambda}$ be the solution of above problem. Therefore, $\tilde{g}_s$ from Step 1 in $\mathcal{P}_{\rm dd}$ is then obtained via $\tilde{g}_s = -\sum_{j=1}^{r}\bar{\lambda}_j\xi_j$.}

All tests are implemented in MATLAB R2022b using IEEE double precision arithmetic and run on a laptop equipped with Intel Core i7, CPU 2.60 GHz and 16 GB of RAM. The quadratic programming solver $\mathbf{quadprog.m}$ in the MATLAB optimization toolbox is used to solve the convex quadratic problem in Step 1 of the Procedure $\mathcal{P}_{\rm dd}$. For all examples, we set the algorithm parameters as follows: $\varepsilon=10^{-4}$, $\delta=10^{-3}$, $c=0.25$, $\alpha=2$, $t_0=1$.

The numerical results are shown in Figs. \ref{fig.1}--\ref{fig.4}. In particular, each picture shows the value space generated by our algorithm for 100 random starting points. In Fig. \ref{fig.21}, we set $p=100$, $m=2$, $q^1=6$, $w^1=(0.1,0.1,0.1,0.2,0.2,0.3)^T$, $q^2=4$, $w^2=(0.1,0.2,0.3,0.4)^T$ and $y_j^i$ is randomly generated for $i=1,2$, $j=1,\cdots,q^i$. 
In Fig. \ref{fig.22}, we set $p=100$, $m=3$, $q^1=6$, $w^1=(0.1,0.1,0.1,0.2,0.2,0.3)^T$, $q^2=4$, $w^2=(0.1,0.2,0.3,0.4)^T$, $q^3=5$, {
	 $w^3=(0.1,0.1,0.2,0.3,0.3)^T$}, and $y_j^i$ is randomly generated for $i=1,2,3$, $j=1,\cdots,q^i$. {
	 In Fig. \ref{fig.31}, we set $p=50$, $m=2$ and randomly generate $A_i$ for $i=1,2$. In Fig. \ref{fig.31}, we set $p=50$ and $m=3$ and $A_i$ are randomly generated for $i=1,2,3$. In Fig. \ref{fig.41}, we set $n=50$, $p=10$, $m=2$, then we randomly generate $A_i$, $b_i$ and $\lambda_i>0$ for $i=1,2$. In Fig. \ref{fig.42}, we set $n=50$, $p=10$, $m=3$, then we randomly generate $A_i$, $b_i$ and $\lambda_i>0$ for $i=1,2,3$.}

\begin{figure}[ht]
	\centering
	\includegraphics[width=2.8in, keepaspectratio]{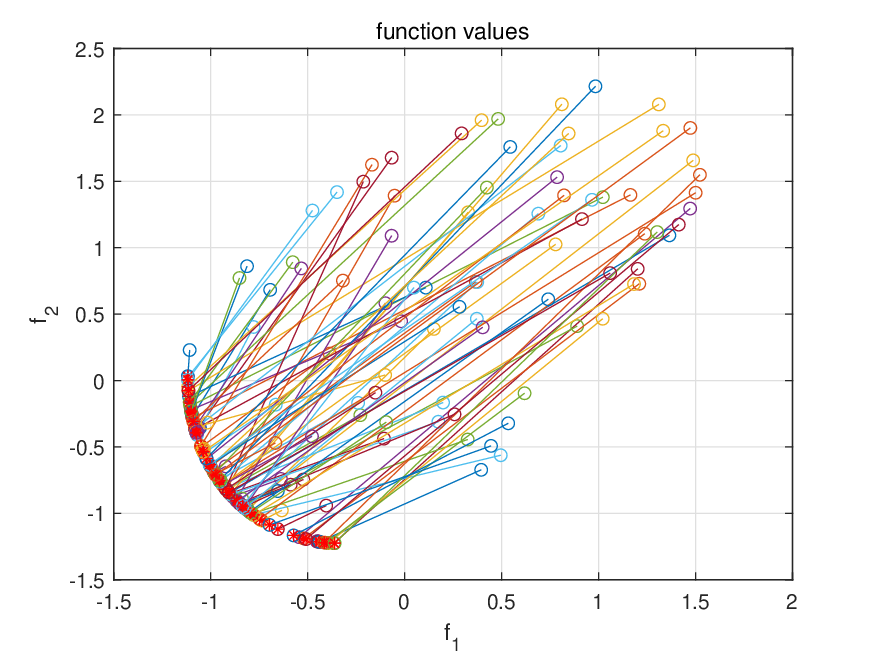}
	\caption{Numerical results for Example \ref{example1}.}
	\label{fig.1}
\end{figure}

\begin{figure}[ht]
	\centering
	\subfloat[$m=2$\label{fig.21}]{
		\includegraphics[scale=0.45]{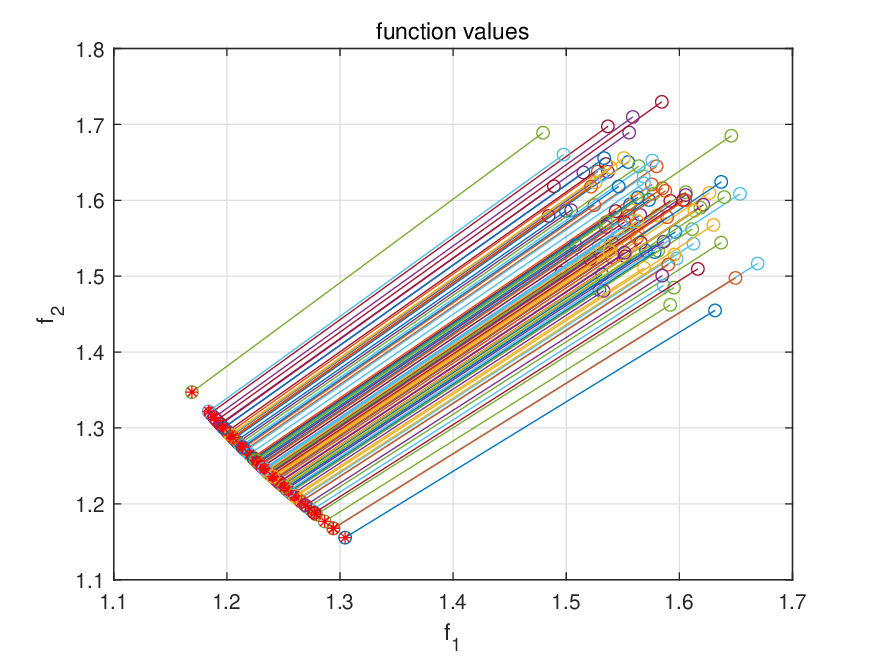}}
	\subfloat[$m=3$\label{fig.22}]{
		\includegraphics[scale=0.45]{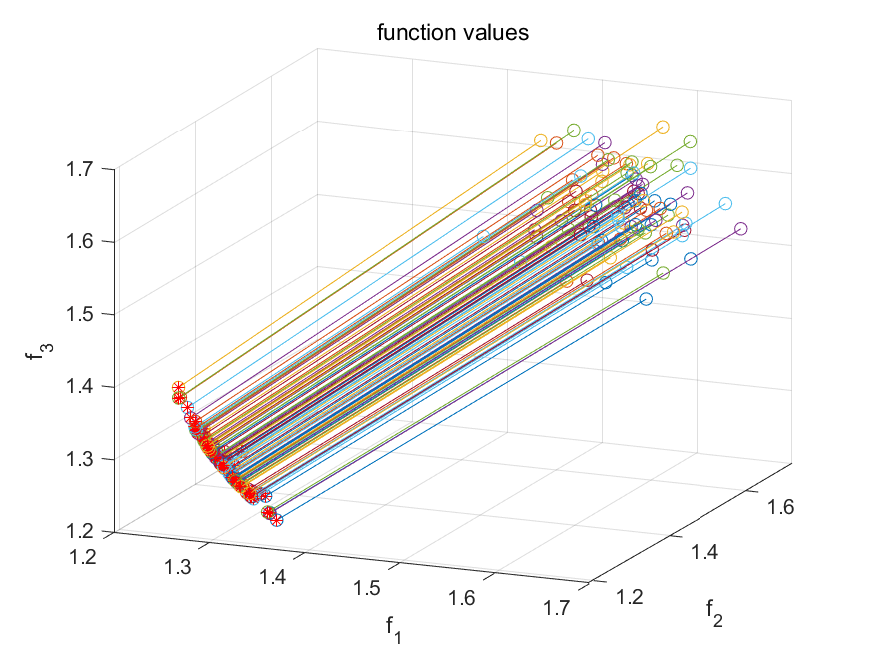}}
		\caption{Numerical results for Example \ref{example2}.}
		\label{fig.2}
\end{figure}

\begin{figure}[ht]
	\centering
	\subfloat[$m=2$\label{fig.31}]{
		\includegraphics[scale=0.45]{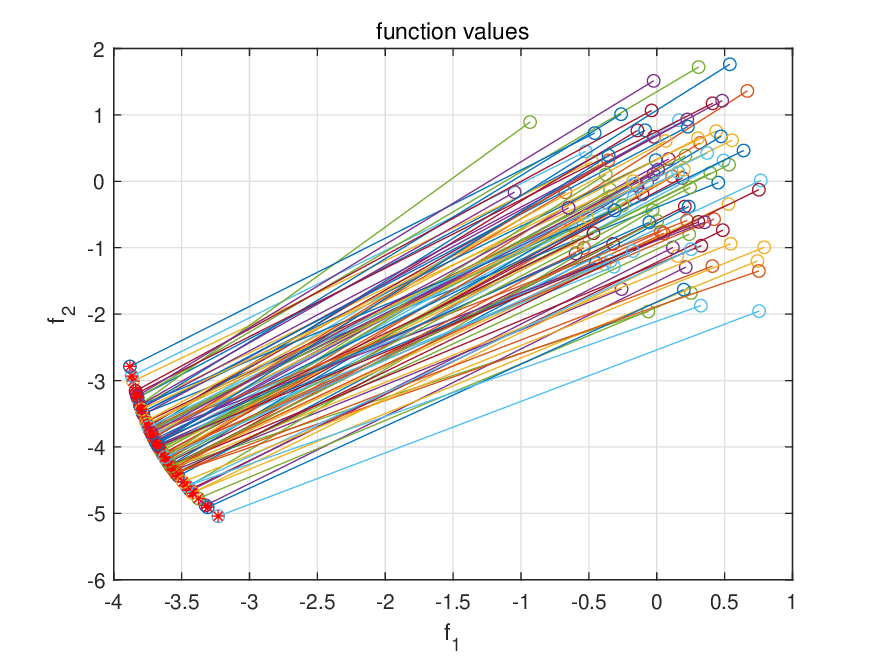}}
	\subfloat[$m=3$\label{fig.32}]{
		\includegraphics[scale=0.45]{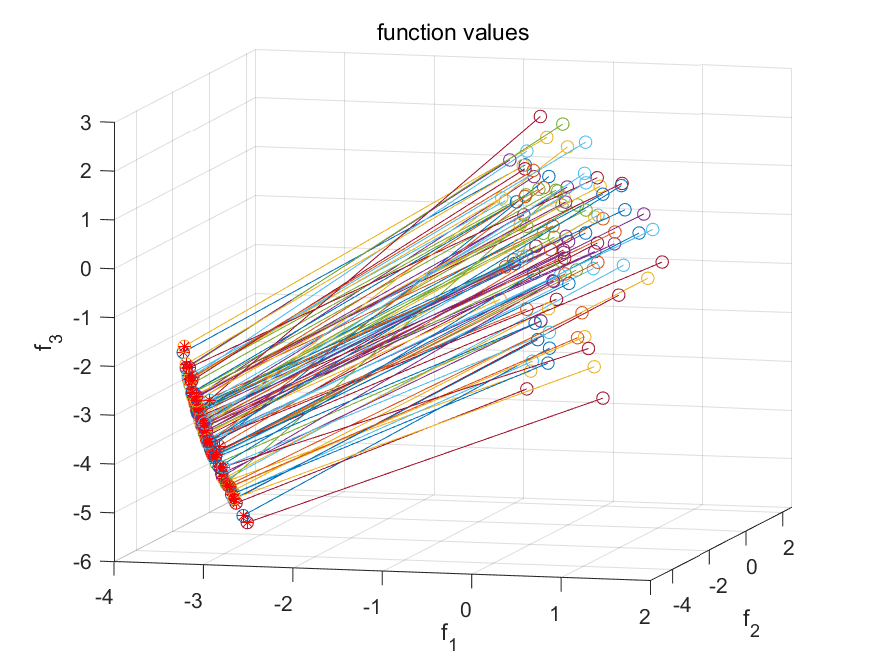}}
	\caption{Numerical results for Example \ref{example3}.}
	\label{fig.3}
\end{figure}

\begin{figure} [ht]
	\centering
	\subfloat[$m=2$\label{fig.41}]{
		\includegraphics[scale=0.45]{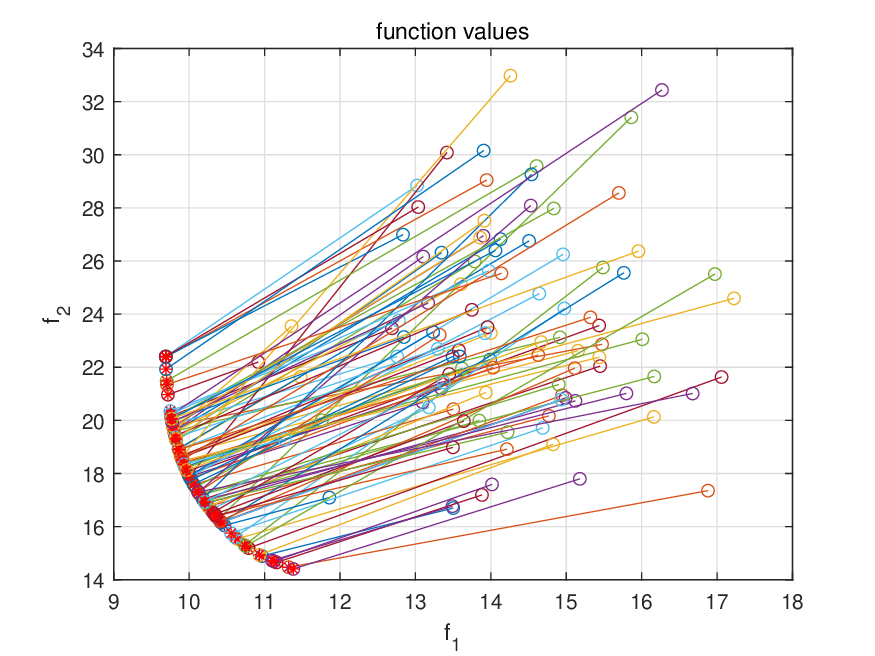}}
	\subfloat[$m=3$\label{fig.42}]{
		\includegraphics[scale=0.45]{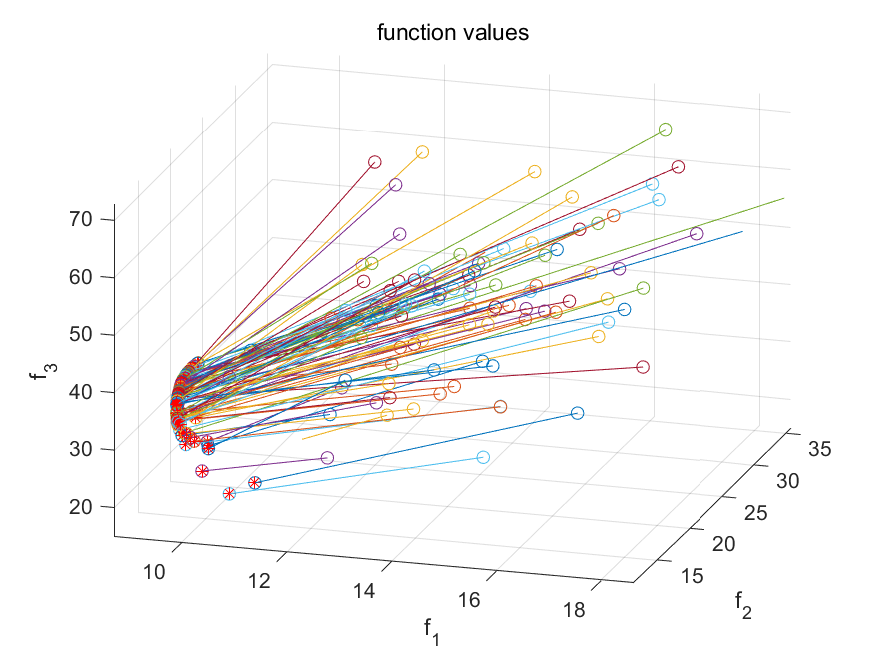}}
	\caption{Numerical results for Example \ref{example4}.}
	\label{fig.4}
\end{figure}

In Figs. \ref{fig.1}--\ref{fig.4}, the hollow points indicate the objective vector values of the initial points, and these marked with red stars are the objective vector values of the final points (namely, the $(\varepsilon,\delta)$-critical points). 
From these figures, we see that nearly all of the final points generated by our method are (approximate) Pareto optimal points for the corresponding examples. Thus, we can obtain the approximations of Pareto sets for the above examples by Algorithm \ref{al1} when some reasonable number of starting points are given. 
Furthermore, in Table \ref{ta1}, we list the average number of iterations (Iter), average number of calls to the objective functions (Nf) and average number of calls to the subgradient of objective functions (Ng) for $100$ random starting points for all examples. In summary, the preliminary numerical results show that our method is effective and promising.
\begin{table}[h]
	\tbl{}
	{\begin{tabular}{lccccccc} \toprule
			Example & \ref{example1} & \ref{example2} ($m=2$)&\ref{example2} ($m=3$) & \ref{example3} ($m=2$)&\ref{example3} ($m=3$) & \ref{example4} ($m=2$)& \ref{example4} ($m=3$) \\ \midrule
			Iter & 4 & 12&13 & 23 &49 & 44 &53 \\
			Nf & 35 & 88&149 & 361 &1168& 975 &1724 \\
			Ng & 9 & 24&40 & 47 &147 & 130 &221 \\ 
         \bottomrule
	\end{tabular}}
	\label{ta1}
\end{table}

\section{Conclusion and discussion}

 In this paper, we have presented an implementable descent method for multiobjective optimization problems with locally Lipschitz components on complete Riemannian manifolds. Our setting is much more general than certain convexities assumed in the existing works. To avoid computing the Riemannian $\varepsilon$-subdifferential of the objective vector function, a convex hull of some Riemannian $\varepsilon$-subgradients is constructed to obtain an acceptable descent direction, which greatly reduces the computational complexity. Furthermore, we extend a necessary condition of the Pareto optimal points to the Riemannian setting. Finite convergence of the proposed algorithm is obtained under the assumption that at least one objective function is bounded below and the employed retraction and vector transport satisfy certain conditions. Finally, some preliminary numerical results illustrate the effectiveness of our method.
 
 {
 In closing, we discuss a way that may enhance the convergence results of our method. 
In Theorem 4.4, we have proved that, for fixed parameters $\varepsilon>0$ and $\delta>0$,  an $(\varepsilon,\delta)$-critical point can be produced by Algorithm 1 after finitely many iterations.
In fact, in the Euclidean setting, based on the idea of \cite{Mahdavi-Amiri,Burkegs} for single objective optimization, 
 the methods in \cite{Gebken B,Gebkenphd} can obtain actual Pareto critical points (i.e., $(0,0)$-critical points) by dynamically reducing $(\varepsilon,\delta)$ to $(0,0)$. 
 As previously mentioned, our method can be seen as a Riemannian version of the method in \cite{Gebken B}. 
  Thus, we believe that some similar stronger convergence results may be obtained by using the strategy that the parameters $\varepsilon$ and  $\delta$ are decreased dynamically to zero, which is a worthy direction for future work.

}


\section*{Acknowledgements}
The authors are sincerely grateful to the Editor and two anonymous referees for their valuable comments and suggestions that improve the original version of this paper significantly.

\section*{Disclosure statement}
No potential conflict of interest was reported by the author(s).
\section*{Funding}
This work was supported by the National Natural Science Foundation of China (12271113, 12171106, 12061013) and Guangxi Natural Science Foundation (2020GXNSFDA238017).

\section*{Notes on contributors}
Chunming Tang (1979) is a Professor in College of Mathematics and Information Science, Guangxi University. He received his bachelor’s degree at Guangxi University (2002) and PhD’s degree in Shanghai University (2008). His research interest lies in nonlinear programming, nonsmooth optimization, and Riemannian optimization.
\vspace{0.3cm}

\noindent Hao He (1997) is a master's student in College of Mathematics and Information Science, Guangxi University. He received his bachelor’s degree at Hefei University of Technology (2020).
His research interest is Riemannian optimization.
\vspace{0.3cm}

\noindent Jinbao Jian (1964) is a Professor in College of Mathematics and Physics, Guangxi Minzu University. He received his bachelor’s degree at Guangxi University (1986) and PhD’s degree in Xi’an Jiaotong University (2000). His research interest lies in optimization theory and its applications.
\vspace{0.3cm}

\noindent Miantao Chao (1981) is an associate Professor in College of Mathematics and Information Science, Guangxi University. He received his bachelor’s degree at Liaocheng University (2004) and PhD’s degree in Beijing University of Technology (2015). His research interest lies in optimization theory and methods.

\end{document}